\documentclass[a4paper,11pt]{article}  

\usepackage{fullpage}
\usepackage{amsmath}
\usepackage{amsfonts}
\usepackage{amssymb}
\usepackage{amsthm}
\usepackage{graphicx}
\usepackage{color}

\usepackage{tikz}

\usetikzlibrary{shapes,backgrounds}
\usetikzlibrary{shadows}
\usetikzlibrary[positioning]
\usetikzlibrary{patterns}
\usetikzlibrary{fadings}
\usetikzlibrary{arrows}
\tikzstyle{vertex}=[inner sep = 3pt,circle,fill=white,draw,minimum size=0.20cm]
\tikzstyle{graphbox}=[inner sep =5pt,rounded corners=10pt, fill = white, draw,minimum size=2.0cm, minimum height=2.0cm]
\tikzstyle{graphcollection} = [inner sep=7pt, drop shadow,fill=lightlightgray,draw,minimum size=2.3cm, minimum height=2.3cm]
\def\rev#1{}

\usepackage[lofdepth,lotdepth]{subfig}

\theoremstyle{plain}
\newtheorem{theorem}{Theorem}
\newtheorem{lemma}{Lemma}

\newtheorem{corollary}{Corollary}
\theoremstyle{definition}

\newcommand{\BoxedGraph}[2]{%
\begin{tikzpicture}
\node[graphbox] (rnum)  at (0,0) {%
\begin{tikzpicture}[-stealth]%
  \foreach \name/\x/\y in {#1}%
    \node[vertex] (G-\name) at (\x,\y) {};
  \foreach \from/\to in {#2}%
    \draw[semithick] (G-\from) -- (G-\to);%
\end{tikzpicture}%
};%
\end{tikzpicture}%
}

 \newdimen\R
\newcommand{\BoxedCircularGraph}[2]{%
\begin{tikzpicture}
\node[graphbox] (rnum)  at (0,0) {%
\begin{tikzpicture}[-stealth]%
\path[use as bounding box] (-0.8,-0.8) rectangle (0.8,0.8); 
  \foreach \name/\x/\y in {#1}%
    \node[vertex] (G-\name) at (\x:\y) {};
  \foreach \from/\to in {#2}%
    \draw[semithick] (G-\from) -- (G-\to);%
\end{tikzpicture}%
};%
\end{tikzpicture}%
}

\newcommand{\BoxedCircularGraphNOBOUND}[2]{%
\begin{tikzpicture}
\node[graphbox] (rnum)  at (0,0) {%
\begin{tikzpicture}[-stealth]%
  \foreach \name/\x/\y in {#1}%
    \node[vertex] (G-\name) at (\x:\y) {};
  \foreach \from/\to in {#2}%
    \draw[semithick] (G-\from) -- (G-\to);%
\end{tikzpicture}%
};%
\end{tikzpicture}%
}

\newcommand{\BoxedCircularGraphNOBOUNDwithcustoarrow}[3]{%
\begin{tikzpicture}
\node[graphbox] (rnum)  at (0,0) {%
\begin{tikzpicture}[-#3]%
  \foreach \name/\x/\y in {#1}%
    \node[vertex] (G-\name) at (\x:\y) {};
  \foreach \from/\to in {#2}%
    \draw[semithick] (G-\from) -- (G-\to);%
\end{tikzpicture}%
};%
\end{tikzpicture}%
}

\usepackage[colorlinks]{hyperref}

\definecolor{darkred}{rgb}{0.5,0,0}
\definecolor{darkgreen}{rgb}{0,0.5,0}
\definecolor{darkblue}{rgb}{0,0,0.5}

\definecolor{gray}{gray}{0.3}
\definecolor{lightlightgray}{gray}{0.95}
\definecolor{lightgray}{gray}{0.7}

 \hypersetup{colorlinks
 ,linkcolor=darkred
 ,filecolor=darkgreen
 ,urlcolor=darkred
 ,citecolor=darkblue}

\usepackage{algorithm}
\usepackage{algorithmic}

\usepackage{xspace}

\newcommand{\dunion}{\mathbin{\dot{\cup}}}

\DeclareMathOperator{\Aut}{Aut}
\newcommand{\swappingpoints}{W}
\DeclareMathOperator{\dist}{dist}

\def\({\bigl(}  \def\){\bigr)}
\let\<=\langle  \let\>=\rangle
\def\nfrac#1#2{{\textstyle\frac{#1}{#2}}}
\def\Aut{\mathrm{Aut}} 
\def\tto{\mathord{\to}}

\def\C{\mathcal{C}}

\title{Switching Reconstruction of Digraphs}
\author{Brendan D.~McKay, Pascal Schweitzer\thanks{This
work is supported 
by the Australian Research Council, the National Research Fund of Luxembourg, and co-funded under the Marie Curie 
Actions of the European Commission (FP7-COFUND). 
Current address for Schweitzer: Forschungsinstitut f\"{u}r Mathematik,
ETH Z\"{u}rich, Switzerland.} \\[2ex]
Research School of Computer Science\\
The Australian National University\\ 
Canberra, ACT 0200, Australia\\
{\tt bdm@cs.anu.edu.au, Pascal.Schweitzer@anu.edu.au}
}
\begin{document}

\maketitle

 \begin{abstract}
   Switching about a vertex in a digraph means to reverse the direction of every edge
   incident with that vertex.
   Bondy and Mercier introduced the problem of whether a digraph can be
   reconstructed up to isomorphism from the multiset of isomorphism types
   of digraphs obtained by switching about each vertex.
   Since the largest known non-reconstructible oriented graphs have
   8 vertices, it is natural to ask whether there are any larger non-reconstructible graphs.
   In this paper we continue the investigation of this question.
   We find that there are exactly 44 non-reconstructible oriented graphs
   whose underlying undirected graphs have maximum degree at most~2.
   We also determine the full set of switching-stable oriented graphs, which
   are those graphs for which all switchings return a digraph isomorphic to the original.
 \end{abstract}

\section{Introduction}

In combinatorics, a reconstruction problem asks whether a combinatorial object can be reconstructed from its ``deck'', where the deck is a multi-set of objects that are slight modifications of the original object. For example, in Kelly and Ulam's reconstruction problem~\cite{MR0087949, MR0120127} the deck of an~$n$-vertex graph~$G$ is the multiset of~$n$ graphs that are each obtained by deleting one vertex of~$G$. Kelly and Ulam conjectured that every graph on at least three vertices is reconstructible from its vertex-deleted deck.
Another way of stating this is that non-isomorphic graphs on at least three vertices have different decks. Following their work, numerous other types of decks and their associated reconstruction problems have been considered. Among them is Stanley's switching reconstruction problem~\cite{MR787322}, where the elements of the deck of a graph are obtained by choosing a vertex~$v$ and then replacing all edges incident with~$v$ with non-edges and vice versa.
As with the previous reconstruction problem, it is conjectured that there are only finitely many graphs that cannot be reconstructed from their deck. For an overview of the various reconstruction variants we refer the reader to the existing surveys~\cite{MR1047783, BondySurvey,MR2096338}.

Recently, Bondy and Mercier~\cite{BondyMercier,Mercier:thesis} introduced further switching reconstruction problems. In their context, the switching deck of a digraph contains all digraphs that are obtained by reversing the orientation of all edges incident with a specified vertex.
In their paper, Bondy and Mercier observe that a subgraph-counting
result of Ellingham and Royle~\cite{MR1152444} applies to switching of connected digraphs. They provide an alternative proof that also applies to disconnected graphs and thereby show that 
for a digraph on~$n$ vertices the number of induced
copies of any digraph on~$k$ vertices is reconstructible whenever $n$ is not of the form $4j$ for $j\le\lfloor k/2\rfloor$.

Bondy and Mercier also list all decks of the oriented graphs on four
vertices and, using this, derive many sets of switching non-reconstructible
digraphs on four and eight vertices. Our computer search has shown that there are, in fact, 5559 non-reconstructible
oriented graphs on 8 vertices. They ask whether there exist non-reconstructible oriented graphs on more than eight vertices, which, by their theorem, must have at least 12 vertices, since the number of vertices must be divisible by four. They also ask which oriented graphs have the property that all switchings are isomorphic. As examples of such digraphs they mention vertex transitive digraphs and switching-stable digraphs, i.e., digraphs for which every switching is isomorphic to the digraph itself.

\paragraph{Our results:}
By an \textit{oriented graph} we mean a digraph without loops, parallel edges, or 2-cycles.  Ignoring the orientations of the edges of an oriented graph gives its \textit{underlying undirected graph}. When we refer to the \textit{degree} of a vertex of an oriented graph, we always mean the degree in the underlying undirected graph; otherwise we will write ``in-degree'' or ``out-degree''.

In this paper we classify all switching-stable oriented graphs. We show that an oriented graph is switching-stable if and only if each of its components has 1, 2, or 4 vertices and each component on 4 vertices is isomorphic to the oriented\rev{P2/26} cycle that has a unique directed path of length 3.

We also determine all non-reconstructible oriented graphs with maximum degree at most~2. There are exactly~44. They yield 29 sets~$\{G,H\}$ of non-isomorphic oriented graphs of maximum degree at most 2 such that~$G$ and~$H$ have the same deck. To determine all these oriented graphs, we prove with combinatorial arguments that there are none on more than 30 vertices and then with an efficient enumeration algorithm using the computer determine all non-reconstructible oriented graphs with at most 30 vertices.

\paragraph{Structure of the paper:} We first formally define the digraph switching reconstruction problem and the related notions (Section~\ref{sec:prelims}). We then classify all switching-stable oriented graphs (Section~\ref{sec:stable:graphs}) and devise some properties of switching-stable sets (Section~\ref{sec:stable:sets}). To classify all non-reconstructible oriented graphs of maximum degree at most 2 we determine all paths (Section~\ref{sec:paths}) and all cycles (Section~\ref{sec:cycles}) that are not reconstructible from their~$t$-decks. An analysis of disconnected non-reconstructible oriented graphs (Section~\ref{sec:disconnected}) then allows us to assemble the statements and determine all non-reconstructible oriented graphs of maximum degree at most~2 (Section~\ref{sec:assemble:all:max:degree:2}).

\section{Preliminaries}\label{sec:prelims}

In this paper we consider labelled graphs which are mostly directed.
For any graph~$G$, we denote by~$\<G\>$ the isomorphism class of~$G$. 
The sign ``='' means equality (never isomorphism).
The automorphism group of~$G$,  denoted~$\Aut(G)$, acts on the vertex set of~$G$, which we will always denote as~$V$.
We denote by~$S_n$ the symmetric\rev{P2/53} group on~$V$.
For~$\gamma,\delta \in S_n$, by~$v^\gamma$ we mean
$\gamma(v)$, by~$v^{\gamma\delta}$ we mean $(v^\gamma)^\delta$ and by~$G^\gamma$ we mean
the graph on the same vertices as~$G$ for which~$v^\gamma\tto w^\gamma$ is an edge
if and only if~$v\tto w$ is an edge of~$G$.

For a digraph~$G$ and a vertex~$v \in V$,~$G_v$ is the \emph{switching of~$G$ at vertex v}, that is, the graph which
equals~$G$ except that the direction of all edges incident with~$v$ is reversed.
$G_{vw}$ means~$(G_v)_w$.
Easy properties are:~$G_{vv}=G$,~$G_{vw}=G_{wv}$, and
$(G_v)^\gamma=(G^\gamma)_{v^\gamma}$, for~$v,w\in V, \gamma\in S_n$.
For a multiset~$W=\{ w_1,\ldots,w_k\}$ where
$w_1,\ldots,w_k\in V$,~$G_W$ means~$G_{w_1\cdots w_k}$.
Clearly, $G_W$ depends only on the parity of the multiplicities of
elements of $W$, and if $W$ is a set (not multiset) we have $G_{V\setminus W}=G_W$.
The property $(G_v)^\gamma=(G^\gamma)_{v^\gamma}$ generalizes to
$(G_W)^\gamma=(G^\gamma)_{W^\gamma}$ for $W\subseteq V, \gamma\in S_n$,
from which it follows that $G_X=G^\gamma$ and~$G_Y=G^\delta$
together imply\rev{P3/17}
$G_{YX^\delta}=G^{\gamma\delta}$. Besides these easy properties, we also frequently need the following observations:

\begin{lemma}\label{lem:switch:one:vertex:observation}
If~$G_W = G^{\delta}$ then~$(G_v)_{Wvv^\delta}= (G_v)^{\delta}$. 
\end{lemma}
\begin{proof}
Using the fact that~$(G_v)^\delta=(G^\delta)_{v^\delta}$, we get~$(G_v)_{Wvv^\delta} = (G_W)_{v^\delta} = (G^{\delta})_{v^\delta} = (G_v)^{\delta}$.
\end{proof}

\begin{lemma}\label{swdiff}
 If~$G$ is a connected oriented graph and~$W,W'\subseteq V$ satisfy
 $G_W=G_{W'}$, then $W'=W$ or $W'=V\setminus W$.
\end{lemma}

\begin{proof} Supposing otherwise, there are adjacent vertices~$v,v'$
such that~$v\in W\bigtriangleup W'$ and~$v'\notin W\bigtriangleup W'$.
Then~$G_W$ and~$G_{W'}$ differ in the orientation of the edge between~$v$ and~$v'$.
\end{proof}

The \emph{deck\/} of a digraph~$G$ is the multiset
$\C(G) = \{ \<G_v\> \mid v\in V \}$.
We say that the digraph~$G$ is \emph{non-reconstructible\/} if there is
a digraph~$H$, not isomorphic to~$G$, such that~$\C(H)=\C(G)$, and
\emph{reconstructible\/} otherwise.
Since it will aid the proof of our main theorem, we generalize this concept as follows.
For~$t \in \mathbb{N}$, the
\emph{$t$-deck\/} of~$G$, 
is the deck of~$G$ augmented by~$t$ copies of~$\<G\>$.
By counting the number of copies of $\<G\>$ that occur in the $t$-decks of
two non-isomorphic digraphs~$G$ and~$H$,\rev{P3/35} we find that there is at
most one~$t$ for which~$G$ and~$H$ have the same~$t$-deck.

\section{Switching-stable digraphs}\label{sec:stable:graphs}

Following~\cite{BondyMercier}, we call a digraph~$D$
\emph{switching-stable\/} if~$D_v$ is isomorphic to~$D$
for all vertices~$v$.
Examples of switching-stable digraphs are shown in Figure~\ref{fig:switching:stable:con:graphs}. The figure shows a graph on one vertex, a graph on two vertices with one edge and a graph that has four vertices~$v_1,v_2,v_3,v_4$ and four edges
$v_1v_2,v_2v_3,v_3v_4,v_1v_4$.
In fact, as the following theorem shows, all switching-stable oriented graphs are obtained from these graphs by forming disjoint unions.

\begin{figure}
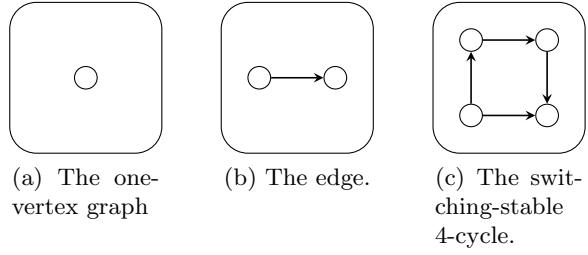


\captionsetup[subfigure]{margin=1pt}

\centering
 \subfloat[The one-vertex graph]{\BoxedGraph {1/0/0}{}}\quad \quad
 \subfloat[The edge.]{\BoxedGraph{1/0/0, 2/1/0}{ 1/2}}\quad \quad
 \subfloat[The~swit\-ching-stable 4-cycle.]{\BoxedGraph{1/0/0, 2/0/1, 3/1/0, 4/1/1}{ 1/2,2/4,4/3,1/3}\label{fig:the:4:cycle:B}}
\caption{All connected switching-stable oriented graphs.}\label{fig:switching:stable:con:graphs}
\end{figure}

\begin{theorem}\label{sstable}
An oriented graph is switching-stable if and only if each of
its components is either an isolated vertex, an isolated edge,
or the oriented 4-cycle that has a unique directed path of length 3, i.e., all components are isomorphic to one of the graphs shown in Figure~\ref{fig:switching:stable:con:graphs}.
\end{theorem}
\begin{proof}
Since a digraph is switching-stable if and only if all of its components are switching-stable, it is sufficient to consider connected oriented graphs.
An easy search among the connected oriented graphs with at most 4 vertices shows that only the three stated examples are switching-stable. Thus, our task is to prove there are none that are larger.

Let~$D$ be a switching-stable oriented graph with $n>4$ vertices and
underlying undirected graph~$G$. 
Let~$S_n$ be the symmetric group on the vertex set~$V$ of~$G$.
Define
\[ \varGamma(D) = \{ \gamma\in S_n \mid D_W=D^\gamma
   \text{~for some~$W\subseteq V$} \}. \]
Since~$\varGamma(D)$ is closed under composition 
($D_X=D^\gamma$ and~$D_Y=D^\delta$ imply
$D_{X^\delta Y}=D^{\gamma\delta}$),~$\varGamma(D)$ is a group.
Note that~$\varGamma(D)$ contains the automorphism group~$\Aut(D)$.

If~$D_X=D_Y$ for~$X,Y\subseteq V$, then either~$X=Y$ or~$X=V\setminus Y$,
by Lemma~\ref{swdiff}.
Therefore, precisely~$2^{n-1}$ different labelled oriented graphs
arise by switching subsets of the vertices in~$D$.
Furthermore, switching preserves the underlying undirected graph~$G$,
so we have
\begin{equation}\label{groups}
  \Aut(D) \le \varGamma(D) \le \Aut(G),
\end{equation}
where~$\Aut(D)$ has index~$2^{n-1}$ in~$\varGamma(D)$.

The number~$2^{n}$ cannot divide~$n!$, while 
$2^{n-1}$ divides $n!$ only when $n$ itself is a power of~2.
Thus we can assume $n=2^k$ for some~$k$ and that
$\Aut(G)$ contains a Sylow 2-subgroup of~$S_n$,
say~$\varLambda$. Every Sylow 2-subgroup of~$S_n$ is an
iterated wreath product of $Z_2$ with
itself (that is, the action on the leaves of a complete
binary tree of the automorphism group of the tree).
This means that~$G$ belongs to a particularly simple class of
graphs, which we now construct.

Consider the sequence of permutation groups
$S_1=\varLambda_0,\varLambda_1,\ldots,\varLambda_k=\varLambda$,
where $\varLambda_0$ acts on a single point and, for $1\le i\le k$,
$\varLambda_i=\varLambda_{i-1}\wr Z_2$ is the standard permutation
representation of the wreath product on $2^i$ points.
That is, $\varLambda_i$ is imprimitive with two blocks of size $2^{i-1}$
that can be exchanged by an element of order 2 and whose
block-wise stabilizer is $\varLambda_{i-1}\times \varLambda_{i-1}$.
This structure of $\varLambda_k\le\Aut(G)$
implies that $G$ consists of two disjoint copies of
some graph on which $\varLambda_{k-1}$ acts, with these two
copies either joined by no edges or by all possible edges.

Applying this logic recursively we find that $G=G_k$ where
$G_0,G_1,\ldots,G_k$ is some sequence of graphs constructed as
follows.

\smallskip\noindent
(a) Graph $G_0$ is the complete graph with one vertex, $K_1$.\\
(b) For $1\le i\le k$, graph $G_i$ is formed by taking two disjoint copies
of $G_{i-1}$ and adding either no more edges or all possible edges
between the two copies.  Note that $G_i$ has $2^i$ vertices and
$\varLambda_i\le\Aut(G_i)$.

\smallskip

Since we are assuming $G$ to be connected, in the final step when
$G_k$ is made from two copies of $G_{k-1}$ we must completely
join the two copies.  Thus $G$ is a regular graph of degree
$d\ge \nfrac12 n$ whose complement is disconnected.

\smallskip
By switching about neighbours of a vertex~$v$ of~$D$, we can
achieve that~$v$ has any out-degree in~$\{0,1,\ldots,d\}$ we choose. Since~$D$ is switching-stable,
this implies that each of the out-degrees~$0,1,\ldots,d$ occurs in~$D$.
Let~$v_0$ be a vertex of out-degree 0.  Switching about~$v_0$
changes the out-degree of~$v_0$ to~$d$, and decreases
the out-degrees of its neighbours by~1.  The only way this
can fail to change the degree sequence of~$D$ is if the
neighbours of~$v_0$ have the distinct out-degrees~$1,2,\ldots,d$.
Now suppose there is a second vertex~$v'_0$ of out-degree~0.
Since~$d\ge \nfrac12 n$,~$v_0$ and~$v'_0$ have at least one common
neighbour, but then switching about~$v'_0$ changes the set
of out-degrees of the neighbours of~$v_0$, which we just proved
is impossible.  So~$v_0$ is the only vertex of out-degree~0.
Similarly,~$D$ has exactly one vertex~$v_d$ of out-degree~$d$,
and its neighbours have out-degrees~$0,1,\ldots,d-1$.

\medskip

Now we divide the argument into two cases depending on whether
or not~$v_0$ and~$v_d$ lie in the same component of the
complement~$\overline G$.
Recall that there are at least two such components.

First assume that~$v_0$ and~$v_d$ lie in different components
of~$\overline G$.
Then every vertex of~$D$ is a neighbour of either~$v_0$ or~$v_d$.
Since~$\Aut(D)$ fixes~$v_0$ and~$v_d$, and their respective
neighbours have unique out-degrees as we showed above, we
have~$|\Aut(D)|=1$.
By~\eqref{groups}, this implies~$\varGamma(D)=\varLambda$.
As explained above, $\varLambda$ has a block-system consisting of two blocks 
$W_0,W_1$ of the same size, and all $\nfrac14 n^2$ edges between $W_0$ and $W_1$ 
are present.
Therefore, for any vertex~$v$,
$D_v=D^{\gamma(v)}$ where~$\gamma(v)\in\varLambda$ satisfies
$\{W_0^{\gamma(v)},W_1^{\gamma(v)}\} = \{W_0,W_1\}$.\rev{P4/48}
Moreover, if~$e$ is the number of edges of $D$
directed from~$W_0$ to~$W_1$, then~$\nfrac14 n^2-e$ is the number of edges
directed the other way. Switching about a vertex must preserve the invariant $\{e,\nfrac14 n^2-e\}$ consisting of the counts of edges directed from one side of the partition to the other.
If~$v\in W_0$ has~$x$ out-neighbours in~$W_1$, after switching about $v$ there are $e+\nfrac12 n-2x$ edges from $W_0$ to~$W_1$, so we must have $e+\nfrac12 n-2x\in\{e,\nfrac14 n^2-e\}$.
Similarly, if~$w$ in~$W_1$ has~$x$ out-neighbours in~$W_0$, after switching about $w$ there
are $e-\nfrac12 n+2x$ edges from $W_0$ to $W_1$, so we must have $e-\nfrac12 n+2x\in\{e,\nfrac14 n^2-e\}$.
These constraints only allow $x$ to have the three values $\nfrac14 n$,
$\nfrac14 n+e-\nfrac18 n^2$ and $\nfrac14 n-e+\nfrac18 n^2$.
However for the vertex~$v_0$, any value for~$x$ from 0 to~$\nfrac12 n$ can be achieved by switching about its neighbours.
Since~$n\ge 8$ (recall that $n$ is a power of~2), this is a contradiction.

Finally, assume that~$v_0$ and~$v_d$ lie in the same component
of~$\overline G$, say the component induced by vertex set~$X$.
If $\overline G$ has exactly two components, then $\Aut(G)$ preserves
them as a blocks and the counting argument we used in the previous
paragraph applies. If the number of components is greater than 2, it
is still a power of 2 (by the construction of~$G$) and $\Aut(G)$
preserves the partition of $V$ into those components.
Switching about a vertex~$w$ not in~$X$ cannot create a 
vertex of out-degree 0 or~$d$ outside~$X$, since all vertices
not in~$X$ are adjacent to~$v_0$ and to~$v_d$.
Therefore, the isomorphism that maps $D$ to $D_v$ preserves
the set $V\setminus X$, and so the subgraph of~$D$ induced
by~$V\setminus X$ is switching-stable. However, it is
connected and has an order which is not a power of 2, which
contradicts what we proved above.
\end{proof}

\section{Switching-stable sets of digraphs}\label{sec:stable:sets}

A set~$M$ of digraphs is called \emph{switching-stable} if every switching of a digraph\rev{Ref 2} in~$M$ is isomorphic to a digraph in~$M$. Examples of switching-stable sets are sets that only contain switching-stable digraphs and sets that contain all orientations of an undirected graph. We now argue that if a small switching-stable set\rev{P5/20} consists of orientations of a large connected graph, then the automorphism group of the graph must be large.

\begin{lemma}
Let~$M$ be a set of orientations of a connected graph~$G$ on~$n$ vertices.
If~$M$ is switching-stable then~$2^{n-1} \leq |M|\, |\Aut(G)|$. 
\end{lemma}

\begin{proof}
Suppose $m=|M|$. Since $M$ is switching-stable, by the pigeonhole principle there must be digraphs~$D,D'\in M$ (possibly equal) such that $D_W$ is isomorphic to~$D'$ for at least~$2^{n} / m$ sets~$W \subseteq V$. Recalling that~$D_W = D_{V\setminus W}$, it follows that for any fixed $v\in V$ there are~$2^{n-1} / m$ sets~$W$ with~$v\in W$ such that~$D_W$ is isomorphic to~$D'$. Thus for each such~$W$ there is an automorphism~$\delta_W$ of~$G$ such that~$D_W = (D')^{\delta_W}$. 
If for two such sets~$W$ and~$W'$ we have~$\delta_W = \delta_{W'}$ then~$D_W =D_{W'}$,
which cannot happen by Lemma~\ref{swdiff}. Thus all automorphisms~$\delta_W$ are distinct, which proves the lemma.
\end{proof}

\begin{corollary}\label{cor:switching:stable:sets:max:deg:2}
Let~$M$ be a set of orientations of a connected graph~$G$ on~$n$ vertices of maximum degree at most~2.
If~$M$ is switching-stable, then~$2^{n-1} \leq 2 n\,|M|$.
\end{corollary}
\begin{proof}
This follows from the previous lemma since every connected~$n$-vertex graph of maximum degree at most~2 has an automorphism group of size at most~$2 n$.
\end{proof}

\section{Paths}\label{sec:paths}

In this section we consider only oriented graphs with an underlying graph that is a path. We characterize all pairs of non-isomorphic oriented paths that have the same~$t$-deck for some~$t\in \mathbb{N}$. The reconstructibility of paths on more than four vertices from their decks was shown by Mercier~\cite{Mercier:thesis}. For our purposes, we need to extend this characterization to $t$-decks. Consequently we cannot assume that the number of vertices of the paths are even and we require an alternative proof. It turns out that the pairs we want to characterize can only consist of oriented paths with three or four vertices. 
\begin{theorem}
Let~$t$ be a non-negative integer. If~$P_1$ and~$P_2$ are two oriented paths on at least 5 vertices that have the same~$t$-deck, then~$P_1$ and~$P_2$ are isomorphic.
\end{theorem}

\begin{proof}
Let~$t$ be a non-negative integer and let~$P$ be an oriented path on at least 5 vertices. We first argue that in the~$t$-deck of~$P$ we can identify the two cards corresponding to switchings about the end vertices and that we can identify the two cards corresponding to switchings about the penultimate vertices.
To show this, we first identify the set~$M$ of cards corresponding to switchings of~$P$
about one of the outer four vertices. We divide the cards into two groups
according to the parity of the number of end edges pointing
outwards. 
If there is a card with both end edges pointing outwards and a card with
both end edges pointing inwards then~$M$ is the group of parity 0. Otherwise M is the group of parity 1.

By a similar argument, considering only cards in~$M$ and the parity of the number of the penultimate edges pointing
outwards, we can distinguish the cards in~$M$ according to whether an end vertex or a neighbour of an end vertex was switched.
This implies that~$P'$, the induced sub-path of~$P$ obtained by deleting the end vertices, is reconstructible.

Suppose~$P$ is a non-reconstructible oriented path on at least 5 vertices. By what we just observed, we can identify a card in the deck of~$P$ in which no leaf and no vertex adjacent to a leaf has been switched. Thus, we can reconstruct the number of end edges of~$P$ pointing outwards. Since we can also identify a card in which an end vertex has been switched, this implies that among the leafs of~$P$ there is one vertex with out-degree 0 and one with in-degree~0. 

Since the oriented path~$P'$ obtained by deleting the end vertices has been determined, two possible reconstructions remain (${\rightarrow} P'{\rightarrow}$ and~${\leftarrow} P' {\leftarrow}$). By considering the switchings about the penultimate vertices, we see that these only have the same deck if $P'$ is the same as its reverse, in which case they are isomorphic.
\end{proof}

\begin{figure}
\centering
\subfloat[Paths that have the same 1-decks.]{
\begin{tikzpicture}
\node[graphcollection] (rnum)  at (0,0) {%
\BoxedGraph {1/0/0, 2/.75/0, 3/1.5/0}{1/2,2/3}\,\,
\BoxedGraph {1/0/0, 2/.75/0, 3/1.5/0}{1/2,3/2}\,\,
\BoxedGraph {1/0/0, 2/.75/0, 3/1.5/0}{2/1,2/3}
};%
\end{tikzpicture}%
}

\subfloat[Paths that have the same decks.]{
\begin{tikzpicture}
\node[graphcollection] (rnum)  at (0,0) {%
\BoxedGraph {1/0/0, 2/.75/0, 3/1.5/0, 4/2.25/0}{1/2,2/3,3/4}\,\,
\BoxedGraph {1/0/0, 2/.75/0, 3/1.5/0, 4/2.25/0}{1/2,3/2,3/4}
};%
\end{tikzpicture}%
}\quad \quad 
\subfloat[Paths that have the same decks.]{
\begin{tikzpicture}
\node[graphcollection] (rnum)  at (0,0) {%
\BoxedGraph {1/0/0, 2/.75/0, 3/1.5/0, 4/2.25/0}{1/2,2/3,4/3}\,\,
\BoxedGraph {1/0/0, 2/.75/0, 3/1.5/0, 4/2.25/0}{2/1,3/2,3/4}
};%
\end{tikzpicture}%
}
\caption{The three families of paths which have the same~$t$-decks. The three paths 
in (a) have the same 1-decks, the other two pairs of graphs have the same decks.}\label{fig:paths}
\end{figure}
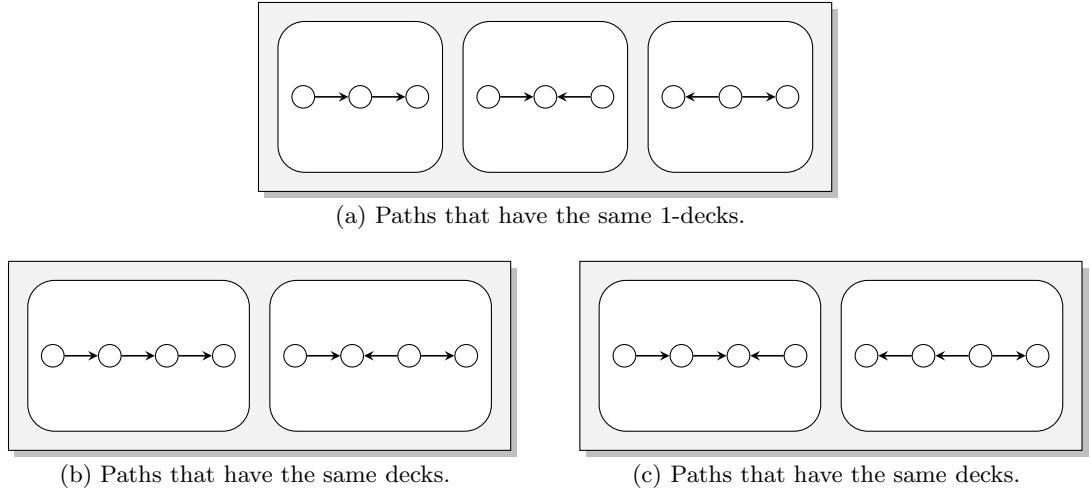

The theorem says that for our purposes we only need to consider oriented paths on at most 4 vertices. Figure~\ref{fig:paths} depicts all families of non-isomorphic oriented paths which have the same~$t$-deck for some~$t\in \mathbb{N}$.

\section{Cycles}\label{sec:cycles}

In this section we consider only graphs whose underlying undirected graph is a cycle. We determine all pairs of non-isomorphic oriented graphs of this type which have the same~$t$-deck for some~$t\in \mathbb{N}$.
The automorphism group of an undirected cycle on~$n$ vertices  is the dihedral group~$D_n$, which consists of~$n$ reflections and~$n$ rotations (the identity is a
rotation).

\begin{lemma}\label{extrot}
  Suppose that for~$n\ge 30$, there are non-isomorphic oriented cycles~$G$ and~$H$ with the same~$t$-deck for some~$t$.
  Then there is a set~$W\subseteq V$ of size 2 or 4 such that~$G_W=G^\gamma$ for some rotation~$\gamma$ of order at least~5.
\end{lemma}

\begin{proof}
For some labelling of~$G$ and~$H$,
the common~$t$-deck of~$G$ and~$H$ can be indexed as
$\{\<J(i)\> : 1\le i\le n+t\}$, so that 
$J(i)=H$ for~$i=1,\ldots,t$,~$J(i)=G$ for~$i=n{+}1,\ldots,n{+}t$,
$\<G_v\>=\<J(v)\>$ for~$i=1,\ldots,n$, and
$\<H_v\>=\<J(v{+}t)\>$ for~$i=1,\ldots,n$.

We will first find 8 sets~$W\subseteq V$ of size~2
such that~$G_W=G^{\gamma(W)}$, where the permutations
$\gamma(W)$ are nontrivial and distinct.  We do this in two
cases, depending on whether~$t\le 4$.

First suppose that~$t\le 4$.
For each~$v=t{+}1,\ldots,n$, there is some~$w$ such that
$G_{vw}\in\<H\>$.  We know that~$w\ne v$ since~$H$ is not
isomorphic to~$G$.
Then there are at least~$n-t-2$ vertices~$x$ of~$G_{vw}$ other than
$v$ and~$w$, such that for some~$y\ne x$,~$G_{vwxy}\in\<G\>$.
By Lemma~\ref{swdiff},~$G_{vwxy}=G^{\gamma(v,x)}$ for some
nontrivial~$\gamma(v,x)\in D_n$.
We had~$(n-t)(n-t-2)$ choices of~$(v,x)$, but the permutations
obtained might not be all different.
By Lemma~\ref{swdiff},~$\gamma(v,x)$ in fact depends injectively
on which elements of the multiset~$\{v,w,x,y\}$ occur an
odd number of times.  Any such multiset with 4 distinct elements
can be obtained from at most 8 choices of~$(v,x)$.  The other
possibility is that~$y=v$ or~$y=w$, so that exactly two elements,
say~$a,b$, of the multiset~$\{v,w,x,y\}$ occur an odd number of times.  Any
such set~$\{a,b\}$ can occur at most once for each value of~$v$,
that is, at most~$n-t$ times.  Finally, there are~$2n-1$ nontrivial
elements in~$D_n$.  If~$k$ is the number of choices~$(v,x)$ for
which we obtain a 2-set, we must have
$8(2n-k-1)+(n-t)k\ge (n-t)(n-t-2)$, which implies~$k\ge 9$ when
$t\le 4$ and~$n\ge 30$.

Next suppose that~$t\ge 5$.  For~$v=1,\ldots,5$,~$G_v$ is isomorphic
to~$H$, then for each such~$v$, there are five values of~$w$, corresponding
to the vertices of~$H$ labelled~$n{+}1,\ldots,n{+}5$, for which~$G_{vw}$
is isomorphic to~$G$.  Discounting the possibility~$v=w$, and noting that
$G_{vw}=G_{wv}$, we find~$10$ sets~$\{v,w\}$ such that 
$G_{vw}=G^{\delta(v,w)}$ for nontrivial distinct~$\delta(v,w)\in D_n$.

The proof so far shows that we have a set $\mathcal{W}$ of 8
distinct sets of size 2 such that for each $W\in\mathcal{W}$,
$G_W=G^{\gamma(W)}$, where the permutations $\gamma(W)\in D_n$ are
nontrivial and distinct.

If, for some $W\in\mathcal{W}$, $\gamma(W)$ is a rotation of
order at least 5, we are done, so assume that is not the case.
We will find such a rotation by combining two sets $U,W$ using
the rule $G_{U^{\gamma(W)}W}=G^{\gamma(U)\gamma(W)}$. However we
need to avoid the case where $U^{\gamma(W)}=W$, since we want the
symmetric difference of $U^{\gamma(W)}$ and $W$ to have size~2
or~4. This means that for each $W\in\mathcal{W}$, there is at
most one $U\in\mathcal{W}$, namely $W^{-\gamma(W)}$, excluded
for this reason. We will call the product $\gamma(U)\gamma(W)$
\textit{acceptable} if $U\ne W^{-\gamma(W)}$.

First suppose that $\mathcal{W}$ contains four rotations. If there
is one, $\gamma_1$, of order~2 or~4, and two $\gamma_2,\gamma_3$, of
order~3, then $\gamma_2\gamma_1$ and $\gamma_3\gamma_1$ have order~6
or~12 and at least one of them is acceptable. The same argument
holds if $\gamma_1$ has order~3 and $\gamma_2,\gamma_3$ have order~2
or~4.

If $\mathcal{W}$ does not contain four rotations, it has at
least 5 reflections~$\sigma_1,\ldots,\sigma_5$. 
The product $\sigma_1\sigma_1$ is trivial and therefore not acceptable. Thus, the four products $\gamma_2 = \sigma_2\sigma_1, \ldots,\gamma_5 =\sigma_5\sigma_1$ are acceptable and in particular nontrivial.
If one of the nontrivial rotations~$\gamma_2,\ldots,\gamma_5$ has 
order at least 5, we are done, so assume that is not the case. By the same argument as before there must be indices $i,j,k \in\{2,\ldots,5\}$ with~$i\neq j$ such that~$\gamma_{i}\gamma_{k}^{-1}$ and $\gamma_j\gamma_k^{-1}$ have order~6 or~12. Noting that~$\gamma_{i}\gamma_{k}^{-1}= \sigma_i \sigma_1 \sigma_1 \sigma_{k} =  \sigma_i \sigma_{k}$ and~$\gamma_{j}\gamma_{k}^{-1}= \sigma_j \sigma_1 \sigma_1 \sigma_{k} = \sigma_j \sigma_{k}$ we conclude that one of the two products must be acceptable and has order at least 5.

One of these two possibilities (4 rotations or 5 reflections) must happen if we have 8 distinct
nontrivial elements of~$D_n$ in~$\mathcal{W}$.  
This shows that there is a set~$W\subseteq V$ of size 2 or 4 such that~$G_W=G^\gamma$ for some rotation~$\gamma$ of order at least~5.
\end{proof}

Let~$\delta$ be a rotation.
We define~$\swappingpoints(G,\delta)$ to be the set~$W$ that satisfies~$G_W = G^{\delta}$ and~$|W| < n/2$, if it exists.
Note that the set~$W$ is unique if it exists. Let~$\dist(G,\delta)$ be the set of pairwise distances (lengths of shortest paths on the undirected cycle) between the elements of~$\swappingpoints(G,\delta)$.
Note that~$|\swappingpoints(G,\delta)| = |\swappingpoints(G,\delta^{-1})|$
and~$\dist(G,\delta)= \dist(G,\delta^{-1})$, since $W(G,\delta^{-1})=W^{\delta^{-1}}$.

\begin{lemma}\label{lem:know:size:w}
Let~$G$ be an oriented cycle and let $\delta$ be a nontrivial rotation
such that $G_W = G^{\delta}$ for some $W\subseteq V$ with $n >2|W|+8$.
Then $\swappingpoints(G_v,\delta)$ is well defined for each $v\in V$. Furthermore,
for any oriented cycle~$H$ that has an extended deck which is also an extended deck of~$G$, the set~$\swappingpoints(H,\delta)$ is well defined and~$|\swappingpoints(H,\delta)| = \max\big\{|W(G_v,\delta)|-2 \mid  v\in V\big\} = |W|$.
\end{lemma}

\begin{proof}
The fact that $\swappingpoints(G_v,\delta)$ is well defined follows from Lemma~\ref{lem:switch:one:vertex:observation}.
If~$H$ is an oriented cycle that has an extended deck which is also an extended deck of~$G$, then there is a vertex~$v$ such that~$G_{v} = H$ or there are two not necessarily distinct vertices~$v,w$ such that~$G_{vw} = H$. Thus, by Lemma~\ref{lem:switch:one:vertex:observation},~$H_{W vv^\delta} = {(G_{v})}_{W vv^\delta}  = (G_{v})^\delta = {H}^{\delta}$ or~$H_{W vv^\delta ww^\delta} = {(G_{vw})}_{W vv^\delta ww^\delta}  = (G_{vw})^\delta = {H}^{\delta}$. Since~$|W| +4 < n/2$, 
in the first case, we have~$W' = \swappingpoints(H,\delta) = W \bigtriangleup \{v,v^\delta\}$, and in the second case
we have~$W' = \swappingpoints(H,\delta) = W \bigtriangleup \{v,v^\delta\}\bigtriangleup \{w,w^\delta\}$. In either case~$\swappingpoints(H,\delta)$ is well defined. 

For all~$v$ it holds that~$|W'|-2\leq |{W'\bigtriangleup \{v,v^\delta}\}|\leq|W'|+2$.
Since~$n> 2(|W|+ 4) \geq 2|W'|$, there exists a~$u$ such that~${W'\bigtriangleup \{u,u^\delta}\}$ has size~$|W'|+2$.  
Thus~$|W'| = \max\big\{|W(G_v,\delta)|-2 \mid  v\in V \big\}$. 
\end{proof}

Note that the fact that~$|\swappingpoints(G_v,\delta)| = |\swappingpoints(G_v,\delta^{-1})|$ for all~$v\in V$ implies that the value of~$\max\big\{|W(G_v,\delta)|-2 \mid  v\in V \big\}$ depends only on the deck. The previous lemma thus says in other words that~$|W|$ can be reconstructed, for any particular~$\delta$.

\begin{lemma}\label{lem:know:distances}
Let~$G$ be an oriented cycle. Suppose there exists a nontrivial rotation~$\delta$ that is not of order 2, such that
for~$W = \swappingpoints(G,\delta)$ we have~$|W| \geq 2$ and~$n>4|W|+4$,  then~$\dist(G,\delta)$ is reconstructible from every extended deck.
\end{lemma}
\begin{proof}
Let~$r$ be the integer that is equal to the distance between a vertex and its image under~$\delta$.
Since~$|W| \geq 2$ and~$n>4|W|+4$ together imply~$n >2|W|+8$, Lemma~\ref{lem:know:size:w} applies and~$|W|$ can be reconstructed.
First note that for any integer~$d$ the value of~$|\{v \mid d \in \dist(G_v,\delta)\}|$ 
depends only on the deck.
\begin{itemize}
\item For~$d\notin  \dist(G,\delta)$ and~$d\neq r$ we have~$|\{v \mid d \in \dist(G_v,\delta)\}| \leq 4 |W|$.
\item For~$d\in  \dist(G,\delta)$ and~$d\neq r$ we have~$|\{v \mid d \in \dist(G_v,\delta)\}| \geq n-4$.
\end{itemize}

These inequalities remain true if the sets on the left hand side range over all cards in an extended deck of~$G$.
Thus if~$4|W| < n-4$, then we can decide whether~$d$ is in~$\dist(G,\delta)$, for~$d\neq r$.
To show that it is also possible to decide~$r\in  \dist(G,\delta)$, we note that~$r\in  \dist(G,\delta)$ if and only if there exist a vertex~$v\in V$ such that~$|\swappingpoints(G_v,\delta)| = |W|-2$.
\end{proof}

\begin{lemma}\label{lem:delta:not:among}
Let~$G$ be an oriented cycle. Suppose there exists a nontrivial rotation~$\delta$ that rotates by~$r< n/2$ positions, such that
for~$W = \swappingpoints(G,\delta)$ we have~$r\notin\dist(G,\delta)$,~$|W|\geq 2$, and~$n>4|W|+4$. Then~$G$ is reconstructible from every extended deck.

\end{lemma}
\begin{proof}

The inequalities~$|W| \geq 2$ and~$n>4|W|+4$ imply that~$n>2|W|+8$. 
By Lemma~\ref{lem:know:size:w}, for any oriented cycle~$H$ that has an extended deck which is also an extended deck of~$G$ we have~$|\swappingpoints(H,\delta)|= |\swappingpoints(G,\delta)|$.
We show the existence of a vertex~$v'$ such that
\begin{enumerate}\itemsep=0pt
\item~$v',{(v')}^{\delta}\notin W$,\label{count:1:l2}
\item~$v'$ is not at distance~$r$ from any vertex in~$W$, and\label{count:2:l2}
\item~${(v')}^{\delta}$ does not have distance~$r$ from any vertex in~$W$.\label{count:3:l2}
\end{enumerate}

We call such a vertex \emph{good}. 
We count the number of vertices that violate one of the conditions:
There are~$|W|$ vertices in~$W$ and~$|W|$ vertices in~$W^{\delta}$, thus there are at most~$2|W|$ vertices that violate Condition~\ref{count:1:l2}.
There are at most~$|W|$  vertices that violate Condition~\ref{count:2:l2} but do not violate Condition~\ref{count:1:l2}.
There are at most~$|W|$  vertices that violate Condition~\ref{count:3:l2} but do not violate Condition~\ref{count:1:l2}.
Thus in total we conclude that at most~$2|W|+|W|+|W|$ vertices are not good.
Since~$n>4|W|$, there is a good vertex.

Let~$v'$ be a good vertex. We show that for all~$v''\neq v'$ the set~$\swappingpoints(G_{v'v''},\delta) = W \bigtriangleup \{v',(v')^{\delta}\} \bigtriangleup \{v'',(v'')^{\delta}\}$ is not of size~$|W|$. Since~$v'$ is good, the set~$W\bigtriangleup \{v',(v')^{\delta}\}= W \cup  \{v',(v')^{\delta}\}$ is of size~$|W|+2$. 
Furthermore, since~$v'$ is good and the distance~$r$ is not in~$\dist(G,\delta)$, the only two vertices in~$W \cup  \{v',(v')^{\delta}\}$ that are distance~$r$ apart are~$v'$ and~$(v')^{\delta}$. Since~$r\neq n/2$, for~$v''\neq v'$ the set~$\{v',(v')^{\delta}\}$ is not equal to the set~$\{v'',(v'')^{\delta}\}$. This shows that~$W \bigtriangleup \{v',(v')^{\delta}\} \bigtriangleup \{v'',(v'')^{\delta}\}$ is not of size~$|W|$.

Let~$H$ be an oriented cycle that has an extended deck which is also an extended deck of~$G$. Consider the card~$G_{v'}$ in this extended deck.
Since~$|\swappingpoints(H,\delta)|= |W|$, we conclude that~$H\neq G_{v'}$ and that~$H \neq (G_{v'})_{v''}$ for all~$v''\neq v'$.
Thus~$H =  (G_{v'})_{v'} = G$.
\end{proof}

\begin{lemma}\label{lem:W:size:4}
Let~$G$ be an oriented cycle on~$n>28$ vertices. If there is a nontrivial rotation~$\delta$ which is not of order~2, such that~$|\swappingpoints(G,\delta)| = 4$ then~$G$ is reconstructible from every extended deck. 

\end{lemma}
\begin{proof}
Let~$r$ be the integer that is equal to the distance between a vertex and its image under~$\delta$. Define~$D= \dist(G,\delta)$. By the previous lemma, we may assume that~$r\in D$.
The conditions on~$W=\swappingpoints(G,\delta)$ imply that~$n>2|W|+8$ and~$n>4|W|+4$.
Thus, by Lemmas~\ref{lem:know:size:w} and~\ref{lem:know:distances}, for any graph~$H$ that has an extended deck which is also an extended deck of~$G$, we have~$|\swappingpoints(H,\delta)| = 4$ and~$\dist(H,\delta)= D$.

We show the existence of a vertex~$v' \in V$ such that
\begin{enumerate}\itemsep=0pt
\item~$v',(v')^{\delta}\notin W$,\label{count:1}
\item~$v'$ is not at distance~$r$ from any vertex in~$W$, and\label{count:2}
\item~for at least three distinct vertices~$x_1,x_2,x_3\in W$ the distance from~$v'$ to~$x_i$ is not contained in~$D$.\label{count:3} 
\end{enumerate}

We call such a vertex \emph{good}. 
We count the number of vertices that violate one of the conditions:
There are~$4$ vertices in~$W$ and~$4$ vertices in~$W^{\delta}$. However, since~$r\in D$, one vertex lies in~$W \cap W^{\delta}$, thus there are at most~$7$ vertices that violate Condition~\ref{count:1}.
Since~$r\in D$, there are at most~$3$  vertices that violate Condition~\ref{count:2} but do not violate Condition~\ref{count:1}.
We now bound the number of vertices for which at least~$2$ of the distances to vertices in~$W$ are contained in~$D$. We double count the ordered pairs~$(x,u)$ with~$x\in W$,~$u\in V$, where $u$ doesn't satisfy Condition~\ref{count:1} or~\ref{count:2} and the distance between~$x$ and~$u$ is in~$D$. There are at most~$|D| \cdot 2\cdot 4 - 4\cdot 3$ such pairs, since there are~$4\cdot 3$ pairs~$(x,u)$  with~$x,u\in W$. Thus there are at most~$(8 |D|-12) / 2$ vertices not in~$W$ with two or less distances to vertices in~$W$ not contained in~$D$. 
Thus in total we conclude that at most~$7+ 3 + 4|D|-6 \leq   28$ vertices are not good.
Since~$n>28$, there is a good vertex.

Let~$v'$ be a good vertex. Since~$v'$ is good, the set~$W\bigtriangleup \{v',(v')^{\delta}\}= W \cup  \{v',(v')^{\delta}\}$ is of size~$|W|+2$. We now show that for all~$v''\neq v'$ the set~$W'' = W \bigtriangleup \{v',(v')^{\delta}\} \bigtriangleup \{v'',(v'')^{\delta}\}$ has two vertices of a distance not contained in~$D$, or the set~$|W''|$ is not of size~$|W|$.

Case 1: If~$(v'')^{\delta} = v'$ then~$v'' \notin W\bigtriangleup \{(v'),(v')^{\delta}\}$, since with the exception of~$(v')^{\delta}$ no vertex in~$W\cup \{(v'),(v')^{\delta}\}$ has distance~$r$ from~$v'$ and~$r\neq n/2$. Thus~$|W''|\neq|W|$.

Case 2: If~$v''\neq v'$ and~$(v'')^{\delta} \neq v'$, then~$v'\in W''$. Three of the vertices in~$W$ are at a distance to~$v'$ which is not contained in~$D$. At least one of these vertices is contained in~$W''$.

Let~$H$ be a graph that has an extended deck which is also an extended deck of~$G$. Consider the card~$G_{v'}$ in this extended deck.
Since~$|\swappingpoints(H,\delta)|= 4$ and~$D= \dist(H,\delta)$, we conclude that~$H\neq G_{v'}$ and that~$H \neq (G_{v'})_{v''}$ for all~$v''\neq v'$.
Thus~$H =  (G_{v'})_{v'} = G$.
 \end{proof}

\begin{lemma}\label{lem:W:size:2}
Let~$W = \swappingpoints(G,\delta)$ where~$\delta$ is a nontrivial rotation of order at least~$5$ rotating by~$r$ positions. If\/~$|W| = 2$,~$n>12$  and~$r\in \dist(G,\delta)$, then~$G$ is reconstructible from every extended deck.
\end{lemma}

\begin{proof}
The assumptions~$|W| = 2$ and~$n>12$ imply that~$n>2|W|+8$.

For the sake of contradiction we assume~$G$ is not reconstructible. 
By the assumptions,~$W$ is of the form~$W = \{w_1,w_2\}$ with~$w_2 = w_1^{\delta}$. Define~$w_0 = w_1^{\delta^{-1}}$ and~$w_3 = w_2^{\delta}$.

Since~$\swappingpoints(G_v,\delta) =  \{w_1,w_2\}\bigtriangleup \{v,v^{\delta}\}$ for any vertex~$v$, it holds that~$|\swappingpoints(G_v,\delta)| = 2$ if and only if~$v\in \{w_0,w_2\}$. Since~$G_{w_1w_2} = G^{\delta}$, we have~$(G_{w_0})^\delta = (G_{w_0})_{w_1w_2 w_0w_1}= G_{w_2}$. 
Thus, all cards~$G_v$ with~$|\swappingpoints(G_v,\delta)| = 2$ are isomorphic. 
We can therefore identify the isomorphism type of the card~$G_{w_2}$ from the deck.
Note that~$\swappingpoints(G_{w_2},\delta) = \{w_1,w_3\}$.

We first show that for every vertex~$v' \notin\{w_1,w_2\}$ the set
~$\swappingpoints(G_{w_2 v'},\delta) = \{w_1,{w_3}\} \bigtriangleup\{v',{{v'}^\delta\}}$ 
is either not of size 2, or it is of size two, but its two vertices are not at distance~$r$: Indeed, whenever~$v'\notin \{w_0,w_1,w_2,w_3\}$, then~$\{w_1,{w_3}\}\bigtriangleup \{v',(v')^{\delta}\}$ is not of size two. If~$v'\in~\{w_0, w_3\}$ then~$\{w_1,{w_3}\}\bigtriangleup \{v',(v')^{\delta}\}$ is of the form~$\{u,u^{\delta^3}\}$ for some vertex~$u\in V$ and thus, since~$\delta$ is of order at least~5, does not contain two vertices at distance~$r$.

Since~$(w_1)^{\delta} =w_2$, we know that~$((G_{w_2})_{w_1})^\delta =G = (G_{w_2})_{{w_2}}$. 
Let~$H$ be a graph that has an extended deck which is also an extended deck of~$G$.
For all~$v' \notin \{w_1,w_2\}$ we have that~$|\swappingpoints(G_{w_2 v'},\delta)| \neq 2$ or~$r\notin \swappingpoints(G_{w_2 v'},\delta)$. Thus, for all~$v' \notin \{w_1,w_2\}$,~$H\neq G_{w_2 v'}$. Since~$\delta$ is not of order~3,~$\dist(G,\delta)\neq  \dist(G_{w_2},\delta)$ and thus~$H\neq G_{w_2}$.
For all~$v' \in \{w_1,w_2\}$ the graph~$G_{w_2{v'}}$ is isomorphic to~$G$. Therefore~$H$ is isomorphic to~$G$. In any case we obtain a contradiction.
\end{proof}

\begin{theorem}\label{thm:at:least:30:vertices:reconstructible}
Every oriented cycle on at least 30 vertices is reconstructible from each of its extended decks.
\end{theorem}
\begin{proof}
Suppose~$G$ is an oriented cycle on~$n\geq 30$ vertices that is not reconstructible from one of its extended~decks. By Lemma~\ref{extrot} there is a set~$W\subseteq V$ of size 2 or 4 such that~$G_W=G^\gamma$ for some rotation~$\delta$ of order at least~5. Suppose~$\delta$ operates by rotating~$G$ by~$r$ positions. If~$r\notin\dist(G,\delta)$ then~$G$ is reconstructible by Lemma~\ref{lem:delta:not:among}.
If~$r\in\dist(G,\delta)$, then~$G$ is reconstructible by Lemma~\ref{lem:W:size:2} if~$|W|=2$ and by Lemma~\ref{lem:W:size:4} if~$|W| = 4$.
\end{proof}

Figures~\ref{fig:3:cycles:with:equiv:t:decks}--\ref{fig:8:cycles:with:equiv:t:decks} show families of non-isomorphic oriented cycles which have the same~$t$-deck for some~$t\in \mathbb{Z}$. In fact Figure~\ref{fig:5:cycles:with:equiv:t:decks} shows two graphs which have the same~$(-1)$-decks, where the~$(-1)$-deck of graph~$D$ is obtained from the deck of~$D$ by deleting~$\langle D\rangle$.
All graphs have been computed using the graph generation package \texttt{nauty}~\cite{McKay:1981} developed by the first author. Together with Theorem~\ref{thm:at:least:30:vertices:reconstructible}, this computation  also demonstrates that there are no other examples.

\begin{figure}[p]
 \hfill
  \begin{minipage}[t]{.40\textwidth}
    \begin{center}  
     \begin{tikzpicture}
     \node[graphcollection] (rnum)  at (0,0) {%
     \BoxedCircularGraphNOBOUND {1/0/\R, 2/120/\R, 3/240/\R}{1/2,2/3,3/1}\,\,
     \BoxedCircularGraphNOBOUND {1/0/\R, 2/120/\R, 3/240/\R}{1/2,2/3,1/3}
     };%
     \end{tikzpicture}%
     \captionof{figure}{The 3-cycles that have the same~$1$-decks.}\label{fig:3:cycles:with:equiv:t:decks}
    \end{center}
  \end{minipage}
  \hfill
  \begin{minipage}[t]{.40\textwidth}
    \begin{center}  
      \begin{tikzpicture}
      \node[graphcollection] (rnum)  at (0,0) {%
      \BoxedCircularGraph {1/0/\R, 2/90/\R, 3/180/\R, 4/270/\R}{1/2,2/3,3/4,4/1}\,\,
      \BoxedCircularGraph {1/0/\R, 2/90/\R, 3/180/\R, 4/270/\R}{1/2,3/2,3/4,1/4}
      };%
      \end{tikzpicture}
      \captionof{figure}{The 4-cycles that have the same decks.}\label{fig:4:cycles:with:equiv:t:decks}
    \end{center}
  \end{minipage}
  \hfill
\end{figure}

\begin{figure}
 \hfill
  \begin{minipage}[t]{.40\textwidth}
    \begin{center}  
     \begin{tikzpicture}
     \node[graphcollection] (rnum)  at (0,0) {%
     \BoxedCircularGraph {1/0/\R, 2/51/\R, 3/103/\R,4/154/\R,5/206/\R,6/257/\R,7/309/\R}{1/2,2/3,3/4,5/4,5/6,7/6,1/7} \,\,
     \BoxedCircularGraph {1/0/\R, 2/51/\R, 3/103/\R,4/154/\R,5/206/\R,6/257/\R,7/309/\R}{1/2,2/3,3/4,5/4,6/5,6/7,1/7}
     };%
     \end{tikzpicture}%
     \captionof{figure}{The 7-cycles that have the same~$1$-decks.}\label{fig:7:cycles:with:equiv:t:decks}
    \end{center}
  \end{minipage}
  \hfill
  \begin{minipage}[t]{.40\textwidth}
    \begin{center}  
      \begin{tikzpicture}
      \node[graphcollection] (rnum)  at (0,0) {%
      \BoxedCircularGraph {1/0/\R, 2/72/\R, 3/144/\R,4/216/\R,5/288/\R}{1/2,3/2,3/4,5/4,5/1}\,\,
      \BoxedCircularGraph {1/0/\R, 2/72/\R, 3/144/\R,4/216/\R,5/288/\R}{1/2,3/2,3/4,4/5,5/1}
      };%
      \end{tikzpicture}
      \captionof{figure}{The 5-cycles that have the same~$(-1)$-decks.}\label{fig:5:cycles:with:equiv:t:decks}
    \end{center}
  \end{minipage}
  \hfill
\end{figure}

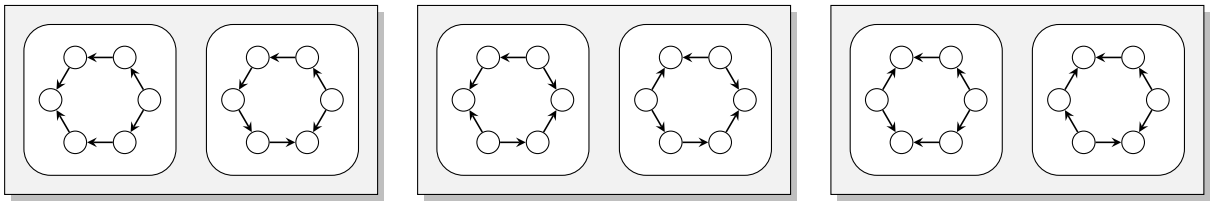
\begin{figure}
\centering
\subfloat 
{
\begin{tikzpicture}
\node[graphcollection] (rnum)  at (0,0) {%
\BoxedCircularGraph {1/0/\R, 2/60/\R, 3/120/\R,4/180/\R,5/240/\R,6/300/\R}{1/2,2/3,3/4,5/4,6/5,1/6} \,\,
\BoxedCircularGraph {1/0/\R, 2/60/\R, 3/120/\R,4/180/\R,5/240/\R,6/300/\R}{1/2,2/3,3/4,4/5,5/6,1/6} 
};%
\end{tikzpicture}%
} \quad 
\subfloat 
{
\begin{tikzpicture}
\node[graphcollection] (rnum)  at (0,0) {%
\BoxedCircularGraph {1/0/\R, 2/60/\R, 3/120/\R,4/180/\R,5/240/\R,6/300/\R}{2/1,2/3,3/4,5/4,5/6,6/1} \,\,
\BoxedCircularGraph {1/0/\R, 2/60/\R, 3/120/\R,4/180/\R,5/240/\R,6/300/\R}{2/1,2/3,4/3,4/5,5/6,6/1} 
};%
\end{tikzpicture}%
} \quad 
\subfloat 
{
\begin{tikzpicture}
\node[graphcollection] (rnum)  at (0,0) {%
\BoxedCircularGraph {1/0/\R, 2/60/\R, 3/120/\R,4/180/\R,5/240/\R,6/300/\R}{1/2,2/3,4/3,4/5,6/5,1/6} \,\,
\BoxedCircularGraph {1/0/\R, 2/60/\R, 3/120/\R,4/180/\R,5/240/\R,6/300/\R}{1/2,2/3,4/3,5/4,5/6,1/6} 
};%
\end{tikzpicture}%
}
\caption{The three families of 6-cycles which have the same~$2$-decks.}\label{fig:6:cycles:with:equiv:t:decks}
\end{figure}

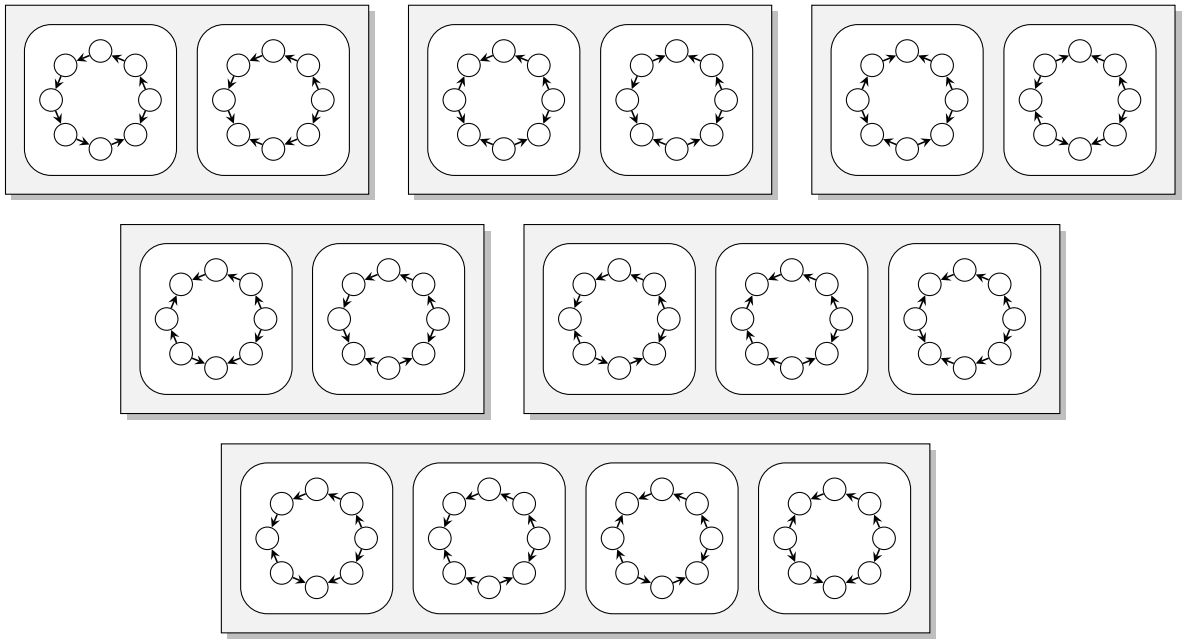
\begin{figure}
\centering
\subfloat 
{
\begin{tikzpicture}
\node[graphcollection] (rnum)  at (0,0) {%
\BoxedCircularGraph {1/0/\R, 2/45/\R, 3/90/\R,4/135/\R,5/180/\R,6/225/\R,7/270/\R,8/315/\R}{1/2,2/3,3/4,4/5,5/6,6/7,7/8,1/8}\,\,
\BoxedCircularGraph {1/0/\R, 2/45/\R, 3/90/\R,4/135/\R,5/180/\R,6/225/\R,7/270/\R,8/315/\R}{1/2,2/3,3/4,4/5,5/6,7/6,8/7,1/8}
};%
\end{tikzpicture}%
} \quad 
\subfloat 
{
\begin{tikzpicture}
\node[graphcollection] (rnum)  at (0,0) {%
\BoxedCircularGraph {1/0/\R, 2/45/\R, 3/90/\R,4/135/\R,5/180/\R,6/225/\R,7/270/\R,8/315/\R}{1/2,2/3,3/4,5/4,5/6,7/6,7/8,1/8}\,\,
\BoxedCircularGraph {1/0/\R, 2/45/\R, 3/90/\R,4/135/\R,5/180/\R,6/225/\R,7/270/\R,8/315/\R}{1/2,2/3,4/3,4/5,5/6,7/6,7/8,1/8}
};%
\end{tikzpicture}%
} \quad
\subfloat
{
\begin{tikzpicture}
\node[graphcollection] (rnum)  at (0,0) {%
\BoxedCircularGraph {1/0/\R, 2/45/\R, 3/90/\R,4/135/\R,5/180/\R,6/225/\R,7/270/\R,8/315/\R}{1/2,2/3,4/3,5/4,5/6,7/6,7/8,1/8}\,\,
\BoxedCircularGraph {1/0/\R, 2/45/\R, 3/90/\R,4/135/\R,5/180/\R,6/225/\R,7/270/\R,8/315/\R}{1/2,2/3,4/3,4/5,6/5,6/7,8/7,1/8}
};%
\end{tikzpicture}%
}

\subfloat 
{
\begin{tikzpicture}
\node[graphcollection] (rnum)  at (0,0) {%
\BoxedCircularGraph {1/0/\R, 2/45/\R, 3/90/\R,4/135/\R,5/180/\R,6/225/\R,7/270/\R,8/315/\R}{1/2,2/3,3/4,5/4,6/5,6/7,8/7,1/8}\,\,
\BoxedCircularGraph {1/0/\R, 2/45/\R, 3/90/\R,4/135/\R,5/180/\R,6/225/\R,7/270/\R,8/315/\R}{1/2,2/3,3/4,4/5,5/6,7/6,7/8,1/8}
};%
\end{tikzpicture}%
} \quad
\subfloat 
{
\begin{tikzpicture}
\node[graphcollection] (rnum)  at (0,0) {%
\BoxedCircularGraph {1/0/\R, 2/45/\R, 3/90/\R,4/135/\R,5/180/\R,6/225/\R,7/270/\R,8/315/\R}{1/2,2/3,3/4,4/5,6/5,6/7,7/8,1/8}\,\,
\BoxedCircularGraph {1/0/\R, 2/45/\R, 3/90/\R,4/135/\R,5/180/\R,6/225/\R,7/270/\R,8/315/\R}{1/2,2/3,3/4,5/4,6/5,7/6,7/8,1/8}\,\,
\BoxedCircularGraph {1/0/\R, 2/45/\R, 3/90/\R,4/135/\R,5/180/\R,6/225/\R,7/270/\R,8/315/\R}{1/2,2/3,3/4,5/4,5/6,7/6,8/7,1/8}
};%
\end{tikzpicture}%
} 

\subfloat 
{
\begin{tikzpicture}
\node[graphcollection] (rnum)  at (0,0) {%
\BoxedCircularGraph {1/0/\R, 2/45/\R, 3/90/\R,4/135/\R,5/180/\R,6/225/\R,7/270/\R,8/315/\R}{1/2,2/3,3/4,4/5,6/5,6/7,8/7,1/8}\,\,
\BoxedCircularGraph {1/0/\R, 2/45/\R, 3/90/\R,4/135/\R,5/180/\R,6/225/\R,7/270/\R,8/315/\R}{1/2,2/3,3/4,4/5,6/5,7/6,7/8,1/8}\,\,
\BoxedCircularGraph {1/0/\R, 2/45/\R, 3/90/\R,4/135/\R,5/180/\R,6/225/\R,7/270/\R,8/315/\R}{1/2,2/3,3/4,5/4,6/5,6/7,7/8,1/8}\,\,
\BoxedCircularGraph {1/0/\R, 2/45/\R, 3/90/\R,4/135/\R,5/180/\R,6/225/\R,7/270/\R,8/315/\R}{1/2,2/3,3/4,5/4,5/6,6/7,8/7,1/8}
};%
\end{tikzpicture}%
} \quad \quad

\caption{The six families of 8-cycles which have the same decks.}\label{fig:8:cycles:with:equiv:t:decks}
\end{figure}

\section{Disconnected graphs}\label{sec:disconnected}

In this section we consider disconnected digraphs. For a digraph~$C$ and a digraph~$G$ let~$n_C(G)$ be the number of components of~$G$ that are isomorphic to~$C$. We call two components of a digraph \emph{switching adjacent} if one can be obtained from the other by switching about a vertex.

\begin{lemma}\label{lem:non:adjacent}
If a digraph~$G$ contains at least two non-isomorphic switching adjacent components, then~$G$ is reconstructible.
\end{lemma}
\begin{proof}
Suppose~$G$ contains the non-isomorphic switching adjacent components~$A$ and~$B$. Let~$H$ be a digraph that has the same deck as~$G$. Since the deck contains a card that has at least two~$A$ components, one of the components of~$H$ is~$A$. By symmetry~$B$ is also a component of~$H$. Thus, the number~$n_A(H)$ of~$A$ components of~$H$ is the maximum number of~$A$ components appearing in a card minus 1 and thus depends only on the deck. By symmetry the number~$n_B(H)$ of~$B$ components in~$H$ also only depends on the deck. There is a card~$X$ with~$n_A(H) +1$ components isomorphic to~$A$ and~$n_B(H) -1$ components isomorphic to~$B$. The graph~$H$, and thus also the graph~$G$, must be isomorphic to the graph obtained from the card~$X$  by replacing a component isomorphic to~$A$ with a component isomorphic to~$B$.
\end{proof}

\begin{lemma}\label{lem:non:stable:more:than:2:comp}
If a digraph~$G$ contains at least three components of which at least two are not switching-stable, then~$G$ is reconstructible.
\end{lemma}
\begin{proof}
Let $H$ be a digraph that has the same deck as~$G$. By the previous lemma we may assume that all non-isomorphic components of~$G$ are not switching adjacent, and similarly for~$H$.
Let~$A$ and~$B$ be components of $G$ that are not switching-stable. Then there is card containing both of them, so at least one, say $A$, is a component of~$H$.
Consider a card $X$ with the fewest components isomorphic to~$A$.
Since $A$ is not switching-stable, this card corresponds to switching $G$ about a vertex in $A$, and similarly for~$H$. Since $A$ is not switching adjacent to any component of~$G$ or~$H$ that is not isomorphic to~$A$, both $G$ and $H$ are the graph obtained by replacing the unique component of $X$ switching adjacent to $A$ by a copy of~$A$.  That is, $G$ and $H$ are isomorphic.
\end{proof}

For a digraph~$G$, a \emph{possible component} is a connected digraph that appears as some component in some digraph~$H$ that has the same deck as~$G$. 
A \emph{definite component} is a connected digraph that appears as a component in every digraph that has the same deck as~$G$.

\begin{lemma}\label{lem:non:stable:exactly:2:comp}
If a digraph~$G$ contains exactly two components of different sizes of which the smaller one is not switching-stable, then~$G$ is reconstructible.
\end{lemma}
\begin{proof}
Suppose the components are~$A$ and~$B$ and~$A$ is of smaller size. Then~$A$ appears in more than half the cards and is thus definite.
Since~$A$ is not switching-stable, there is a card~$CB$ with~$C$ non-isomorphic to~$A$. Since~$A$ is definite and~$B$ cannot be a switching of~$A$ this implies that~$B$ is definite.
\end{proof}

\begin{lemma}\label{lem:comp:sam:size:switching:stable:set}
If~$G$ is a digraph that consists of two components with the same number of vertices then~$G$ is reconstructible or the possible components form a switching-stable set of size at most~4.
\end{lemma}

\begin{proof}
Suppose~$G = A+B$. By Lemma~\ref{lem:non:adjacent} we may assume that~$A$ and~$B$ are not simultaneously non-isomorphic and switching adjacent.

We observe that every possible component appears at least~$n/2$ times in the deck. Moreover every switching of a possible component is in the deck.
We distinguish cases according to the number of possible components of~$G$.

\emph{4 possible components:} If~$G$ has exactly 4 non-isomorphic possible components, then they form a switching-stable set.

\emph{3 possible components:}  If~$G$ has exactly 3 non-isomorphic  possible components then two of them must be adjacent. If there is a reconstruction that has two isomorphic components, then this component appears~$n$ times in the deck. The other two possible components appear~$n/2$ times, so the three components form a switching-stable set.
Suppose now~$G$ does not have isomorphic components. Since~$G$ does not have adjacent components and does not have isomorphic components, it is not the case that all 3 possible components are adjacent. If one of the~3 possible components is adjacent to both other components, then~$G$ must be the graph that consists of the non-adjacent components. Otherwise suppose~$A$ and~$B$ are adjacent and~$C$ is not adjacent to either of them. Since~$A$ and~$B$ both appear at least~$n/2$ times and~$C$ cannot switch to~$A$ or~$B$,~$A$ must always switch to~$B$ and~$B$ must always switch to~$A$. Thus if~$A$ appears in the deck then~$G= B+C$. Otherwise~$G= A+C$. In any case~$G$ is reconstructible.

\emph{2 possible components:} If~$G$ has exactly 2 non-isomorphic  possible components~$A$ and~$B$, then~$G$ is isomorphic to\rev{P13/15} $A+B$ if they are not adjacent. We can thus assume that~$G\neq A+B$ by Lemma~\ref{lem:non:adjacent}. Thus,~$G=A+A$ or~$G=B+B$. 
Either the set~$\{A,B\}$ is switching-stable or the deck contains a card of the form~$A+C$ or a card of the form~$B+C$ with~$C\notin\{A,B\}$, implying that one of the components~$A$ or~$B$ is definite.

\emph{1 possible component:} If~$G$ has exactly 1 non-isomorphic  possible component then~$G$ is isomorphic to the disjoint union of this component with itself.
\end{proof}

\begin{theorem}
If~$G$ and~$H$ are disconnected non-isomorphic digraphs with the same deck, then
\begin{enumerate}
\item $G$ and~$H$ each have exactly two components, and the possible components of~$G$ and~$H$ are of the same size and form a switching-stable set of size at most 4, or \label{thm:item:switchingstable}
\item $G$ and~$H$ each have\rev{P13/T4} exactly one component that is not switching-stable, and these two components have the same~$t$-deck for some~$t\in \mathbb{N}$.
\end{enumerate}
\end{theorem}

\begin{proof}
We first argue that if~$G$ and~$H$ both have exactly one component that is not switching-stable then
the two non-switching-stable components in~$G$ and~$H$ have the same~$t$-deck for some~$t\in \mathbb{N}$. The number of switching-stable components of each isomorphism type in every card of a deck is invariant over all cards and the same for the original graph. Let~$t$ be the number of vertices contained in switching-stable components of~$G$. This implies that~$t$ is also the number of vertices contained in switching-stable components of~$H$. The multi-set of non-switching-stable components in cards of~$G$ is equal to the~$t$-deck of~$G$. The analogous statements holds for~$H$ which shows that the two non-switching-stable components in~$G$ and~$H$ have the same~$t$-deck.

If~$G$ and~$H$ have more than 3 components then by Lemma~\ref{lem:non:stable:more:than:2:comp} they both have exactly one component that is not switching-stable, showing that the second option holds.

If~$G$ and~$H$ have 2 components, then either by Lemma~\ref{lem:non:stable:exactly:2:comp} the second option holds or by Lemma~\ref{lem:comp:sam:size:switching:stable:set}  the possible components  of~$G$ form a switching-stable set of size at most~$4$. By definition, all components of~$H$ are possible components of~$G$. 
\end{proof}

We now describe all oriented graphs with maximum degree at most~2 that satisfy Property~\ref{thm:item:switchingstable} from the previous theorem. Let~$G$ and~$H$ be a pair of such oriented graphs. Let~$c$ be the size of a largest possible component of~$G$ then by Corollary~\ref{cor:switching:stable:sets:max:deg:2} we have~$2^{c-1} \leq 4 \cdot  c\cdot 2$, which implies~$c\leq 6$. 

Note that every oriented path on~$n$ vertices can be switched into every other oriented path on~$n$ vertices. Since there are more than~$4$ oriented paths on 5 vertices, an oriented path contained in a switching-stable set of size at most~$4$ has at most 4 vertices. Figure~\ref{fig:unioins:of:paths} shows the unions of oriented paths that have Property~\ref{thm:item:switchingstable} from the previous theorem.

\begin{figure}
 \hfill
  \begin{minipage}[t]{.45\textwidth}
    \begin{center}  
     \begin{tikzpicture}
     \node[graphcollection] (rnum)  at (0,0) {%
     \BoxedGraph {1/0/0, 2/.75/0, 3/1.5/0, 4/2.25/0, 5/0/0.75, 6/0.75/.75, 7/1.5/0.75, 8/2.25/0.75}{1/2,2/3,3/4,5/6,7/6,7/8} \,\,
     \BoxedGraph {1/0/0, 2/.75/0, 3/1.5/0, 4/2.25/0, 5/0/0.75, 6/0.75/.75, 7/1.5/.75, 8/2.25/.75}{1/2,3/2,4/3,6/5,7/6,7/8}
     };%
     \end{tikzpicture}%
     \captionof{figure}{Disjoint unions of paths which have the same decks.}\label{fig:unioins:of:paths}
    \end{center}
  \end{minipage}
  \hfill
  \begin{minipage}[t]{.45\textwidth}
    \begin{center}  
      \begin{tikzpicture}
      \node[graphcollection] (rnum)  at (0,0) {%
      \BoxedGraph {1/0/0, 2/0.75/0, 3/0.75/0.75, 4/0/0.75, 5/1.5/0, 6/2.25/0, 7/2.25/0.75, 8/1.5/0.75}{1/2,2/3,3/4,4/1,5/6,7/6,7/8,5/8}\,\,
      \BoxedGraph {1/0/0, 2/0.75/0, 3/0.75/0.75, 4/0/0.75, 5/1.5/0, 6/2.25/0, 7/2.25/0.75, 8/1.5/0.75}{1/2,2/3,4/3,1/4,5/6,6/7,8/7,5/8}
      };%
      \end{tikzpicture}
      \captionof{figure}{Disjoint unions of cycles which have the same decks.}\label{fig:unioins:of:cycles}
    \end{center}
  \end{minipage}
  \hfill
\end{figure}

Checking which pairs of cycles of length up to 6 are non-reconstructible can be done by hand or by a computer search. Figures~\ref{fig:unioins:of:paths} and~\ref{fig:unioins:of:cycles} show all pairs of oriented graphs of maximum degree at most 2 that satisfy Property~\ref{thm:item:switchingstable} from the previous theorem. 

\section{Non-reconstructible graphs of maximum degree at most 2}\label{sec:assemble:all:max:degree:2}

We assemble the results from the previous sections to characterize all non-reconstructible graphs with an underlying graph of maximum degree at most 2.

\begin{lemma} Let~$G$ and~$H$ be two digraphs. Let~$G' \dunion S_G$ and~$H' \dunion S_H$ be the decomposition of the graphs obtained by splitting off the parts~$S_G$ and~$S_H$ that contain all switching-stable connected components. Then~$G$ and~$H$ have the same deck if and only if~$S_G$ and~$S_H$ are isomorphic graphs on~$t$ vertices say, and~$G'$ and~$H'$ have the same~$t$-deck.
\end{lemma}

\begin{proof}
The lemma follows directly from the observation that the~$t$-decks of~$G'$ and~$H'$ are obtained by removing all switching-stable components in each card in the deck of~$G$ and~$H$, respectively. 
\end{proof}

\begin{theorem}
The pairs~$\{G,H\}$ of non-isomorphic oriented graphs with  maximum degree at most 2 which have the same deck are exactly the following.

\begin{enumerate}\itemsep=0pt
\item The oriented graphs obtained from paths in Figure~\ref{fig:paths} (a) by adding an isolated vertex.

\item Two oriented paths on 4 vertices that are in one of the two families in
 Figure~\ref{fig:paths} (b) and~(c).

\item The oriented graphs obtained from the two 3-cycles shown in Figure~\ref{fig:3:cycles:with:equiv:t:decks} by adding an isolated vertices.

\item The two oriented  4-cycles in Figure~\ref{fig:4:cycles:with:equiv:t:decks}.

\item The oriented graphs obtained from the two 6-cycles that are in one of the three pairs in Figure~\ref{fig:6:cycles:with:equiv:t:decks} by adding either two isolated vertices or an isolated edge.

\item The oriented graphs obtained from the two 7-cycles shown in Figure~\ref{fig:7:cycles:with:equiv:t:decks} by adding an isolated vertex.

\item Two oriented 8-cycles that are in one of the six families of Figure~\ref{fig:8:cycles:with:equiv:t:decks}.

\item The oriented forests shown in Figure~\ref{fig:unioins:of:paths}.

\item The unions of oriented cycles shown in Figure~\ref{fig:unioins:of:cycles}.
\end{enumerate}

In summary, on 4 vertices there are 4 families of size 2 and 1 family of
size 3.  On 8 vertices there are 13 families of size 2, 1 family of size 3,
and 1 family of size 4.

\end{theorem}

\section{Concluding remarks}\label{sec:final}

As we mentioned earlier, all pairs of non-isomorphic
oriented graphs with the same deck have an order which is a multiple
of~4, and the largest known have 8 vertices~\cite{BondyMercier}.  These include 20 families of
2 tournaments, 4 families of 3 tournaments,
and 2 families of 4 tournaments with the same deck.
One of the latter families is shown in Figure~\ref{fig:tournaments}.

\begin{figure}
 \hfill
  \newdimen\MiddleR
     \MiddleR=1.4cm
    \begin{center}  
		\begin{tikzpicture}
		\begin{scope}
		\node[graphcollection] (rnum)  at (0,0) {%
		\BoxedCircularGraphNOBOUNDwithcustoarrow {1/0/\MiddleR, 2/45/\MiddleR, 3/90/\MiddleR,4/135/\MiddleR,5/180/\MiddleR,6/225/\MiddleR,7/270/\MiddleR,8/315/\MiddleR}
		{1/8,8/2,8/3,8/4,8/5,8/6,7/8,2/1,3/1,4/1,5/1,6/1,7/1,2/3,2/4,2/5,6/2,2/7,3/4,5/3,3/6,7/3,4/5,4/6,7/4,6/5,5/7,7/6}{stealth}\,\,
        \BoxedCircularGraphNOBOUNDwithcustoarrow {1/0/\MiddleR, 2/45/\MiddleR, 3/90/\MiddleR,4/135/\MiddleR,5/180/\MiddleR,6/225/\MiddleR,7/270/\MiddleR,8/315/\MiddleR}
        {8/1,2/8,8/3,8/4,8/5,8/6,7/8,2/1,1/3,1/4,1/5,1/6,1/7,3/2,4/2,5/2,2/6,7/2,3/4,5/3,3/6,7/3,4/5,4/6,7/4,6/5,5/7,7/6}{stealth}\,\,
		\BoxedCircularGraphNOBOUNDwithcustoarrow {1/0/\MiddleR, 2/45/\MiddleR, 3/90/\MiddleR,4/135/\MiddleR,5/180/\MiddleR,6/225/\MiddleR,7/270/\MiddleR,8/315/\MiddleR}
		{8/1,8/2,3/8,8/4,8/5,8/6,7/8,1/2,3/1,1/4,1/5,1/6,1/7,3/2,2/4,2/5,6/2,2/7,4/3,3/5,6/3,3/7,4/5,4/6,7/4,6/5,5/7,7/6}{stealth}\,\,
		\BoxedCircularGraphNOBOUNDwithcustoarrow {1/0/\MiddleR, 2/45/\MiddleR, 3/90/\MiddleR,4/135/\MiddleR,5/180/\MiddleR,6/225/\MiddleR,7/270/\MiddleR,8/315/\MiddleR}
		{8/1,8/2,8/3,4/8,8/5,8/6,7/8,1/2,1/3,4/1,1/5,1/6,1/7,2/3,4/2,2/5,6/2,2/7,4/3,5/3,3/6,7/3,5/4,6/4,4/7,6/5,5/7,7/6}{stealth}
		};%
		\end{scope}
		\end{tikzpicture}%
     \captionof{figure}{A family of four 8-vertex tournaments that have the same decks.}\label{fig:tournaments}
    \end{center}
\end{figure}
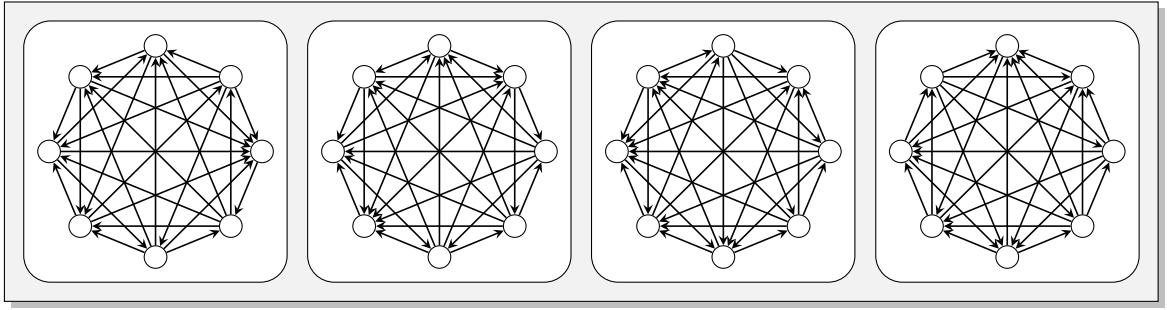

It is too expensive to test by exhaustive enumeration whether all 12-vertex oriented graphs are
determined by their decks. However, we performed a partial search on that size,
finding that all
154,108,311,168 tournaments, all 506,454,795 orientations of graphs
with maximum degree at most~3, and all 16,895,298,180 orientations
of quartic graphs are determined by their decks.

In this paper we have restricted ourselves to oriented graphs, since
a cycle of length 2 is not changed by a switching.
Nevertheless, Figure~\ref{fig:12:cycles:digraphs:with:equiv:decks} shows that new, structurally different reconstruction problems emerge when 2-cycles are allowed.
The two graphs shown have the same deck.  No similar pairs of
cycles occur on 13--20 vertices.

\begin{figure}
 \hfill
  \newdimen\LargeR
     \LargeR=1.25cm
    \begin{center}  
     \begin{tikzpicture}
     \node[graphcollection] (rnum)  at (0,0) {%
     \BoxedCircularGraphNOBOUND
     {1/0/\LargeR, 2/30/\LargeR, 3/60/\LargeR,4/90/\LargeR,5/120/\LargeR,6/150/\LargeR,7/180/\LargeR,8/210/\LargeR,9/240/\LargeR,10/270/\LargeR,11/300/\LargeR,12/330/\LargeR}
     {1/2,2/1,3/2,3/4,4/3,5/4,5/6,6/5,6/7,7/8,8/7,8/9,9/10,10/9,11/10,11/12,12/11,12/1} \,\,
     \BoxedCircularGraphNOBOUND
     {1/0/\LargeR, 2/30/\LargeR, 3/60/\LargeR,4/90/\LargeR,5/120/\LargeR,6/150/\LargeR,7/180/\LargeR,8/210/\LargeR,9/240/\LargeR,10/270/\LargeR,11/300/\LargeR,12/330/\LargeR}
     {1/2,2/1,2/3,3/4,4/3,5/4,5/6,6/5,6/7,7/8,8/7,8/9,9/10,10/9,11/10,11/12,12/11,1/12}
     }; %
     \end{tikzpicture}%
     \captionof{figure}{These 12-cycles contain 2-gons, depicted as two headed arrows. They have the same decks.}\label{fig:12:cycles:digraphs:with:equiv:decks}
    \end{center}
\end{figure}
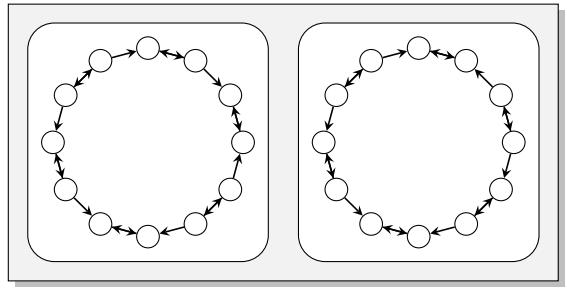

\bibliographystyle{abbrv}

\end{document}